\documentclass[12pt]{article}
\usepackage[utf8]{inputenc}
\usepackage[margin=1.5cm]{geometry}
\usepackage{amsmath,amssymb,amsthm}
\usepackage{comment}
\usepackage{stmaryrd}
\def\Z{\mathbb{Z}}

\DeclareMathOperator{\Mod}{mod}
\DeclareMathOperator{\Tr}{Tr}
\DeclareMathOperator{\rk}{rk}

\def\cR{\mathcal{R}}
\def\bF{\mathbb{F}}

\usepackage{xcolor}
\usepackage{cancel}
\usepackage[normalem]{ulem}
\usepackage{ytableau}

\newtheorem{thm}{Theorem}[section]
\newtheorem{pro}[thm]{Proposition}
\newtheorem{lem}[thm]{Lemma}
\newtheorem{cor}[thm]{Corollary}

\newtheorem{rem}[thm]{Remark}

\usepackage{hyperref}

\title{Matrix Kloosterman sums modulo prime powers}
\author{M. Erd\'elyi, \'A. T\'oth, G. Z\'abr\'adi}
\date{\today}

\begin{document}

\maketitle

\begin{abstract}
We give optimal bounds for matrix Kloosterman sums modulo prime powers extending earlier work of the first two authors on the case of prime moduli. These exponential sums arise in the theory of the horocyclic flow on $GL_n$.
\end{abstract}

\section{Introduction}

The purpose of this paper is to give good upper bounds for the sums
\begin{equation}\label{eq:K_n}
  K_n(A,B; p^k)=\sum_{X\in GL_n(\Z/p^k\Z)} 
\psi( (AX+X^{-1}B)/p^k), 
\end{equation}
with given $A,B\in \mathbb{Z}^{n\times n}$ where for an \(n\times n\) matrix  \(X\) we let \(\psi(X)=e^{2 \pi i \Tr X}\). Note that \(\psi(XY)=\psi(YX)\). A good upper bound may mean different things, it could be optimal, or somewhat crude but easily  usable, and we will provide both. 

These sums themselves are of independent interest. They arise naturally in certain equi-distribution problems and are natural analogues of the classical Kloosterman sums \(\sum_{x\ (p)} e^{2\pi i (ax+bx^{-1})/p}\). 

In our earlier paper \cite{E-T} we dealt with the case \(k=1\). As usual upper bounds modulo a prime require the heavy machinery of some type of Weil cohomology. For higher prime powers the methods are usually of a very different sort, based on Taylor expansions, and occasionally referred to as the stationary phase \cite{Dabrowksi-Fisher}. One such example is provided by Salie's explicit evaluation of the one-dimensional case in \cite{Salie}. We will prove that such a result holds generically even for \(n\times n\) matrices, but generic here is much more restricted than being non-zero, or even invertible.

We summarize the main results. Clearly if \(p\) divides all matrix entries of \(A\) and \(B\) then one may clear appropriate powers and either arrive at a trivial sum, or one where one of \(A\) or \(B\) is different from the zero-matrix. From now on we will assume that this is the case and denote it as
\[
\gcd(A,B,p)=1.
\]

First we have the following reductions to a counting argument.

\begin{pro}\label{pro:red} 
Assume \(\gcd(A,B,p)=1\) and that \(k>1\). 
\begin{enumerate}
    \item If \(k = 2l\) then
    \[
    K_n(A,B; p^k)=
    \sum_{X}  
\psi( (AX+X^{-1}B)/p^k)
    \]
    where the sum is over \(X\in GL_n(\Z/p^k\Z)\) satisfying \(XAX\equiv B \mod p^l\).
    
    \item If \(k=2l+1\) then
     \[
    K_n(A,B; p^k)=\frac{1}{p^{n^2}} 
    \sum_{X} \psi( (AX+X^{-1}B)/p^k)S_{A,B}(X;p)
    \]
    where the sum is over \(X\in GL_n(\Z/p^k\Z)\), \(XAX\equiv B \mod p^l\), and where \[ S_{A,B}(X;p) = \sum_{U \Mod p} \psi((SU+TU^2)/p) 
    ,\] 
   with \(S=(AX-X^{-1}B)/p^l\) and \(T=AX\). 
    
\end{enumerate}
\end{pro}
\begin{rem} 
Note that if \(p \neq 2\), then \(S_{A,B}(X;p)\) is either 0 or equals \(G(AX;p)\), for the generalized Gauss sum
    \[
    G(T;p)=\sum_{U \, \Mod \,p} \psi(TU^2/p)
    \]
See below in Section \ref{sec:background}.    \end{rem}
The case when \(A\) is invertible modulo \(p\) is special and can be made more explicit, in complete analogue of Salie's evaluation of the classical one dimensional Kloosterman sum \cite{Salie}.
\begin{cor}\label{cor:Salie-type}
\begin{enumerate}
\item \(K_n(A,B;p^k)=0\) unless the invariant factors, (the Smith normal forms), of \(A\) and \(B\) agree up to \(\Mod p^l\), where \(l = [k/2]\). 

In particular if \(k>1\) and \(\gcd(\det A,p)=1\) (ie.\ $A$ is invertible mod $p$) then \(K_n(A,B;p^k)=0\) unless \(\gcd(\det B,p)=1\) as well. 

\item Assume that \(\gcd(\det(AB),p)=1\) and that \(AB\) is regular semisimple \(\Mod p\), (i.e. all eigenvalues are different). Then
\begin{enumerate}
    \item If \(k = 2l\) then
    \[
    K_n(A,B; p^k)=
    p^{k n^2/2}\sum_{Y} 
\psi( 2Y/p^k)
    \]
    where the sum is over \(Y\in GL_n(\Z/p\Z)\), \(Y^2\equiv AB \mod p^k\).
    \item If \(k = 2l+1\) then
     \[
    K_n(A,B; p^k)=
    \zeta p^{k n^2/2}\sum_{Y} 
\psi(2Y/p^k)
    \]
    where the sum is over \(Y\in GL_n(\Z/p\Z)\), \(Y^2\equiv AB \mod p^k\), and where \(\zeta\) is a \(p\)-th root of unity.
 \item In particular we have \(|K_n(A,B; p^k)| \leq 2^n p^{k n^2/2} \).   
\end{enumerate}

\end{enumerate}

\end{cor}

Note that in the regular semisimple case we have square root cancellation. Also note that  we do not assume that the eigenvalues of \(AB\) are defined over \(\mathbb{F}_p\). Finally, this condition is generic, its complement is a Zariski closed set.

\medskip
We now return to the non-generic cases. By Proposition \ref{pro:red} in order to bound the sums in \ref{eq:K_n} we need to bound 
\begin{equation}\label{eq:N^*(A,B)-def}
N^*(A,B;p^l) =  \# \{X \in GL_n(\Z/p^l\Z)|AX\equiv X^{-1}B \Mod p^l \}\ .
\end{equation}
 
Note that \[ N^*(A,B;p^l) = N^*(XA,BX^{-1};p^l)\] for any \(X\in GL_n(\Z/p^l\Z)\) and so if \(N^*(A,B;p^l)\) is not zero, then it equals \(N^*(C,C;p^l)\) for \(C=XA\), for any \(X\) for which \(XA\equiv BX^{-1}\pmod{p^l}\).

For \(l=1\) it is possible (see Thm.\ \ref{thm:quadr-matrix-explicit} below) to describe \(N^*(C,C;p)\) explicitly. This allows one to show that for  a well defined exponent \(e=e_C\) we have \(m_n<N^*(C,C;p)/p^e<M_n\) for some absolute constants \(m_n,M_n\) that depend on \(n\) only. The exponent \(e=e_C\) itself depends on the combinatorial type of the Jordan decomposition of \(C\) over an algebraic closure of \(\mathbb{F}_p\).

As a first step towards this goal we have the following reduction.

\begin{pro}\label{prop:primer-decomp}
For any $C\in \mathbb{F}_p^{n\times n}$ let $m_{C^2}(x)=\prod_{j=1}^r f_j(x)^{k_j} \in\mathbb{F}_p[x]$ be the minimal polynomial of $C^2$, where the $f_j(x)\in \mathbb{F}_p[x]$ are irreducible. Let \(V_j=\ker f_j(C^2)^{k_j}\) and \(C_j=C_{|V_j}\) be the restriction of \(C\) to \(V_j\). Then we have
\begin{equation}\label{eq:counting-primer-decomposition}
N^*(C,C;p)=\prod_{j=1}^r N^*(C_j,C_j;p)\ .
\end{equation}
\end{pro}

In case of a primary minimal polynomial the explicit counting formulas depend on the value of \(\left(\frac{x}{f(x)}\right)\)  where as usual \cite{Rosen} \(\left(\frac{x}{f(x)}\right)\) is the quadratic residue symbol, defined for an irreducible polynomial \(f\), 
\[
\left(\frac{g(x)}{f(x)}\right)=\begin{cases} 
1 & \text{ if } g \text{ is a non-zero square in } \mathbb{F}_p[x]/(f), \\
-1 & \text{ if } g \text{ is a not a square in } \mathbb{F}_p[x]/(f), \\
0 & \text{ if } f|g\ .\end{cases} 
\]

To state our main result we need to introduce further notation the details of which are presented in Section \ref{sec:background}. For any partition \(\lambda=[n_1,...,n_k]\) we let \(N_\lambda\) be a nilpotent matrix with Jordan blocks of size \(n_1,...,n_k\). For \(q=p^d\) let \(\mathbb{F}_q\) be the field with \(q\) elements, and   
\(Z_{GL_{|\lambda|}(\mathbb{F}_q)}(N_\lambda)\) be the centralizer of  \(N_\lambda\) in the group \(GL_{|\lambda|}(\mathbb{F}_q)\) where $|\lambda|:=n_1+\dots+n_k$. (See Proposition \ref{prop:centralizer}.)

If we have partitions \(\mu, \nu\) we let \(\lambda=\mu+\nu\) be their join. Also the dual partition \(\lambda'\) of \(\lambda\) may be defined via the matrix \(N_\lambda\) as \(\lambda'=[d_1,...,d_k]\), \(d_1\geq d_2\geq \dots\geq d_k\) where \[\dim \ker N_\lambda^j=d_1+...+d_j.\]

\begin{thm}\label{thm:quadr-matrix-explicit} 
Assume that \(p\neq 2\). Let \(C\) be a  \(m \times m\) matrix  and assume that the minimal polynomial of \(C^2\)  is of the form \(f(x)^k\), where \(f(x)\in \mathbb{F}_p[x]\) is irreducible. We let \(q=p^d\), \(d={\deg f}\), and \(\lambda\) be the partition of $m/d$ with dual \(\lambda'=[d_1,...,d_k]\), \(d_1\geq d_2\geq \dots\geq d_k\) where 
\[
\frac{1}{d}
\dim_{\mathbb{F}_p} \ker f^j(C^2)=d_1+...+d_j.
\]

\begin{enumerate}
\item If \(\left(\frac{x}{f(x)}\right)=1\) we have \(N^*(C,C,p)= N_+^*(\lambda,q)\), where \(q=p^{\deg f}\) and   
\begin{equation}\label{eq:sqrt-count-of-square}
N^*_+(\lambda,q)=\sum_{\lambda=\mu+\nu} 
\frac{\# Z_{GL_{|\lambda|}(\mathbb{F}_q)}(N_\lambda)}{\# Z_{GL_{|\mu|}(\mathbb{F}_q)}(N_\mu) \# Z_{GL_{|\nu|}(\mathbb{F}_q)}(N_\nu)}\ .
\end{equation}
{\color{red} }

\item If \(\left(\frac{x}{f(x)}\right)=-1\) then all \(d_j\) will be even, and $\lambda=\mu+\mu$ for some partition $\mu$. Then we have \(N^*(C,C,p)= N^*_-(\lambda,q)\) where  
\begin{equation}\label{eq:sqrt-count-of-nonsquare}
N^*_-(\lambda,q)=\frac{\# Z_{GL_{|\lambda|}(\mathbb{F}_q)}(N_\lambda)}{\# Z_{GL_{|\mu|}(\mathbb{F}_{q^2})}(N_\mu)}\ .
\end{equation}

\item If \(f(x)=x\) we have 
\begin{equation}\label{eq:counting-nilpotent}   
 N^*(C,C;p)=
p^{\sum_{j=1}^k d_j^2}\prod_{j=1}^k\prod_{t_j=1}^{r_j}\left(1-\frac{1}{p^{t_j}}\right)
\end{equation}
where we put $r_j:=d_j-d_{j+1}$ ($j=1,\dots,k$ where $d_{k+1}:=0$), ie.\ $r_j$ is the number of blocks of size $j\times j$ ($j=1,\dots,k$) in the Jordan normal form of $C$. 
\end{enumerate}

\end{thm}

We are now ready to state our main bounds. 
For a refined statement we need the stable rank of a matrix \(A\) defined as \(\rk_\infty A = \lim_{m\to \infty} \rk A^m\). 

\begin{thm}\label{thm:quadr-matrix-bound}
Assume that \(p\neq 2\) and \(l\geq 1\).
\begin{enumerate} 
    \item Let \(r=\rk (C \Mod p)\), \(r_\infty=\rk_\infty (C \Mod p)\), \ 
\[
N^*(C,C;p^l)\leq 
2^{ r_\infty} p^{ e(l,n,r,r_\infty)}
\]
where
\[
e(l,n,r,r_\infty)= (n-r)^2 + (r-r_\infty)^2 + r_\infty^2/2 +(l-1)\left((n-r)(n-r_\infty) + r_\infty^2/2\right).
\]
\item Assume that \(r=\rk (A \Mod p)=\rk (B \Mod p)>0\). We have
\begin{equation}\label{thm:1.6-part2}
 N^*( A,B; p^l) \leq 
2^{r}
\begin{cases}
p^{ln(n-r)} & \text{ if } \ r\leq n/2\\
p^{ln^2/2 } & \text{ if } \ n/2<r\leq n
\end{cases}.   
\end{equation}

\item In particular if \(n>1\) and \( \gcd(A,B,p)=1\) then \(N^*(A,B;p^l) \leq 2^n p^{l(n^2-n)}\). If \( \gcd(\det A,\det B,p)=1\) then 
\(N^*(A,B;p^l) \leq 2^n p^{ln^2/2}\). 
\end{enumerate}
\end{thm}

\begin{rem}
The case \(p=2\) is special, in view of (\ref{eq:Sylvester-improved-bound}). Assume 
\( AX=X^{-1}B=C \Mod p \). If \(C \neq I_n\), in particular, if \( AB\neq I_n \Mod 2 \),   the bound \(2^{l(n^2-n)}\) still holds.
\end{rem}
We also need the following general bound for the sum \(S_{A,B}(X;p)\).

\begin{pro}\label{pro:matrix-gauss}
Assume that \(A,B\) are such that there exists \(X\) so that \(AX\equiv X^{-1}B \Mod p^l\). Let 
\[ S_{A,B}(X;p) = \sum_{U \Mod p} \psi((SU+TU^2)/p) 
   \ ,\] 
   with \(S=(AX-X^{-1}B)/p^l\) and \(T=AX \Mod p\). Assume that \(p \neq 2\). We have that
\[
|S_{A,B}(X;p)| \leq p^{(n-r)(n-r_\infty)+r_\infty^2/2}\ ,
\]
where \(r =\rk T\), \(r_\infty = \rk_\infty T\).  In particular, we always have $|S_{A,B}(X;p)| \leq p^{n^2-n}$ and under the additional assumption that $A$ is invertible, we have $|S_{A,B}(X;p)| \leq p^{n^2/2}$.
\end{pro}

In view of Proposition~\ref{pro:red} as a corollary of the above we have the following

\begin{thm}\label{thm:main}
Assume that \(n>1\) and the matrices \(A,B\) are not both \(0 \Mod p\). We then have the following bounds.

\begin{enumerate}
    \item If $k=1$, by \cite{E-T} \[|K_n(A,B;p)|\leq \begin{cases}
        2p^{n^2-n+1} & \text{ for all } A,B\\
        4p^{3n^2/4} & \text{ if } \gcd(\det A,\det B,p)=1\\
        4p^{n^2/2} & \text{ if } \gcd(\det B,p)=1, \text{ and } AB^{-1} \text{ regular semisimple}.
    \end{cases}\]
    \item If \(k>1\) we have  
    \[|K_n(A,B;p^k)|\leq 2^n \begin{cases}
        p^{kn^2-\lceil\frac{k}{2}\rceil n} & \text{ for all } A,B\\
        p^{kn^2-\lceil\frac{k}{2}\rceil n^2} & \text{ if } \gcd(\det A,\det B,p)=1.
    \end{cases}\]
\end{enumerate}
\end{thm}

The paper is organized as follows. First in section 2 we prove Proposition \ref{pro:red}, this then gives the optimal bounds for the generic situation. In the next section we list some facts concerning partitions, the Sylvester equation and multivariable Gauss sums. These will be used in the following sections. First in Section \ref{sec:counting-bound-mod-p} we give 
upper bounds for the number of solutions of various quadratic equations in matrices modulo a prime. In the last section Section \ref{sec:general-bounds} we then prove the estimates in the last three statements above.  

While this work was in progress El-Baz, Lee and Strömbergsson \cite{E-B-L-S} independently arrived to quantitavely similar bounds in their work on the equidistribution of rational points on horocycles. While there are some overlaps the main results are different in nature.

\subsection{Acknowledgements}
Erd\'elyi was supported by the MTA–RI Lendület ``Momentum'' Analytic Number Theory and Representation Theory Research Group, and by NKFIH Research Grants FK-127906 and K-135885.

Tóth was supported by by the R\'enyi Institute Lend\"ulet Automorphic Research Group, by the NKFIH Research Grants K 119528 and K-135885.

Zábrádi was supported by the R\'enyi Institute Lend\"ulet Automorphic Research Group, by the MTA–RI Lendület ``Momentum'' Analytic Number Theory and Representation Theory Research Group, and by the NKFIH Research Grants FK-127906 and K-135885.

We thank the referee for a careful reading and helpful comments. We are also grateful to the authors of the paper \cite{E-B-L-S} who provided valuable suggestions for improvements in the presentation of the paper. We also thank them for many stimulating conversations during a Heilbronn research workshop on Effective equidistribution in homogeneous dynamics, that they organized at the University of Bristol.

\section{Reduction to counting}

In this section we prove Proposition \ref{pro:red} and its corollaries. As in the statement we need to deal with the case of even and odd exponents separately.

\subsection{The case \(p^{2l}\)}\hfill

Let \(k=2l\). For any unit \(U \in GL_n(\Z/p^k\Z)\) 
\(K_n(A,B; p^k)=
\sum_{X\in GL_n(\Z/p^k\Z)}
\psi((AXU+U^{-1}X^{-1}B)/p^k).\)

Let \(H < GL_n(\Z/p^k\Z)\) be the subgroup of matrices \(U\) such that \(U\equiv I \ (p^l)\). Explicitly we have
\[H=\{  I+U_1p^l \,|\, U_1 \ (p^l)\}\ .
\]
Since \((I+U_1p^l)^{-1} =I-U_1p^l\), we have  that
\[
K_n(A,B; p^k)=
\frac{1}{p^{l n^2}} 
\sum_{X \in GL_n(\Z/p^k \Z)}
\psi((AX+BX^{-1})/p^k)\sum_{U_1 \Mod p^l}
\psi((AXU+BUX^{-1})/p^l))\ .
\]
As noted above, \(\psi(BUX^{-1})= \psi(X^{-1}BU)\), and so the inner sum
\[
\sum_{U_1 \Mod p^l}
\psi((AX-X^{-1}B)U/p^l))
\]
vanishes, unless \(AX\equiv X^{-1}B \Mod p^l\). This proves the first claim in Proposition~\ref{pro:red}.

\subsection{The case \(p^{2l+1}\)}\hfill

Let \(k=2l+1\). We again use the subgroup \(H\) defined above which in this case consists of matrices \(U=I+U_1p^l+U_2p^{2l}\) where $U_1$ (resp.\ $U_2$) runs on $(\Z/p^l\Z)^{n\times n}$ (resp.\ on $(\Z/p\Z)^{n\times n}$), with inverse \[U^{-1}=I-U_1p^l+(U_1^2-U_2)p^{2l}.\] 
Therefore
\[
K_n(A,B; p^k)=\frac{1}{p^{n^2(l+1)}}\sum_{X\in GL_n(\Z/p^k\Z)}\sum_{U}
\psi((AXU+X^{-1}BU^{-1})/p^k)
\]
where \(U=I+U_1p^l+U_2p^{2l}\) is such that \(U_1\) will run \(\Mod p^l\) and \(U_2\) will run \(\Mod p\).

Now fix \(X \in GL_n(\Z/p^k\Z)\), and consider
\begin{align*}
\sum_{U}
\psi((AX(I+U_1p^l+U_2p^{2l}) + X^{-1}B(I-U_1p^l+(U_1^2-U_2)p^{2l}))/p^k)=\\
\psi( (AX+BX^{-1})/p^k) S_1(X)S_2(X)
\end{align*}
where 
\[
S_1(X)=\sum_{U_1 \ (p^l)} \psi((AX-X^{-1}B)U_1 +X^{-1}BU_1^2p^l)/p^{l+1})\ ,
\]
and
\[
S_2(X)=\sum_{U_2 \ (p)} \psi((AX-X^{-1}B)U_2/p)\ .
\]

Note that \(S_2=0\) unless \(AX\equiv X^{-1}B \Mod p\) in which case $S_2=p^{n^2}$.  
\(S_1\) is a Gauss sum in matrices, albeit a very special one. By the condition from \(S_2\) we have that \(AX-X^{-1}B=pM\) for some integral matrix \(M\), then we have
\[
S_1(X)=\sum_{U_1 \ (p^l)} \psi((MU_1 +TU_1^2)/p^l)\ ,
\]
where \(T \equiv X^{-1}Bp^{l-1}\equiv AXp^{l-1} \Mod p^l\). This gives the claim when \(l=1\). For \(l>1\) note that in view of \(pT \equiv 0 \Mod p^l\) we have for any \(V\) that

\begin{align*}
\sum_{U_1\ (p^l)} \psi((MU_1 +TU_1^2)/p^l)=
\sum_{U_1 \ (p^l)} \psi((M(U_1+pV) +T(U_1+pV)^2)/p^l)=\\
\psi(MV/p^{l-1})\sum_{U_1 \ (p^l)} \psi((MU +TU^2)/p^l)\ .
\end{align*}
and so that \(S_1(X)=\psi(MV/p^{l-1})S_1(X)\).
A suitable choice of \(V\) shows that \(S_1(X)=0\) unless \(M\equiv 0 \mod p^{l-1}\). In the original matrices \(A,B\) this is equivalent to \(AX \equiv X^{-1}B \Mod p^l\) in which case $S_1=p^{(l-1)n^2}S_{A,B}(X;p)$. This gives the second claim of Proposition \ref{pro:red}.

\subsection{The regular semisimple case.}

\begin{proof}[The proof of Corollary~\ref{cor:Salie-type}]
   
Note that if \(X\in GL_n(\Z/p^k\Z)\), \(AX\equiv X^{-1} B \Mod p^l\), and \(A\) is invertible then so is \(B\). Moreover if \(Y=AX\) then 
\(Y^2\equiv AB \Mod p^l\). Assume now that  \(Y^2\equiv AB \Mod p^l\) and let \(X=A^{-1}Y\).
Then \(AX\equiv Y \Mod p^l\) and \(X^{-1}B\equiv Y^{-1}AB\equiv Y \Mod p^l\). If \(AB=Y^2\) is regular semisimple then all the eigenvalues of \(Y\) are different and  no two of them sum to 0. This is exactly the condition (see subsection~\ref{subsec:Sylvester}) to modify \(Y\) by adding a suitable \(p^l Z\) in such a way that \(Y^2 \equiv AB \) hold 
\( \Mod p^k \) as well. In this case 
\[
K_n(A,B;p^k)=K_n(Y,Y;p^k)
\]
and the claim is an easy corollary of the calculations done in previous two subsections and the regularity of 
\(Y^2\). 
\end{proof}

\section{Technical background}\label{sec:background}

\subsection{Partitions}

A partition of an integer \(n\) is an ordered set \(\lambda=[n_{1}, n_{2}, \ldots, n_{r}]\), \(n_1\geq n_2 \geq \dots\geq n_r>0\), of integers satisfying \(\sum_i n_i =n\). We will write \( n=|\lambda|\). If \(\lambda\) and \(\mu\) are two partitions, \(\lambda+\mu\) is the partition obtained by taking the parts of \(\lambda\) and \(\mu\) together (and ordering them). We denote by \([n]\) the partition with one part \(n\). In general if a number \(j\) appears \(r_j\) times in \(\lambda\), the sequence \([...,j,...,j,...]\) will be replaced by \([...,j^{r_j},...]\), so for example the partition with \(n\) parts all equal to 1 is written as \([1^n]\).

Given a partition \(\lambda\) its associated Young (Ferrer) diagram has \(r\) rows with \(n_1,n_2,...n_r\) boxes in each row. For example for \(\lambda=[4,3,1]\) we have the diagram
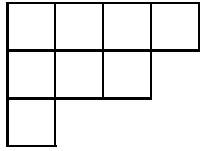
\begin{center}
    \ydiagram{4,3, 1}
\end{center}
The transpose of the diagram of \(\lambda\) is also a Young diagram of a partition \(\lambda'\) called the conjugate or dual partition to \(\lambda\) which may be described as follows. Let \(r_{i}=r_i(\lambda)\) be the number of parts of \(\lambda\) which are equal to \(i \geq 1\) and \(d_{i}=\sum_{j \geq i} r_{j}.\) Then 
\begin{equation}\label{eq:dual-partn}
\lambda^{\prime}=\left[d_{1}, d_{2}, \ldots\right]   
\end{equation}
which has the diagram
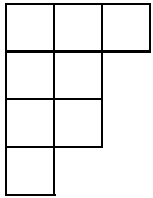
\begin{center}
    \ydiagram{3,2,2,1}
\end{center}
in our example \(\lambda=[4,3,1]\).


\subsection{Centralizers in \(GL_n(\mathbb{F}_q)\)}

At first assume $p\neq 2$ and let \(N\) be the nilpotent transformation of an \(\mathbb{F}_q\) vector space \(V\) of dimension \(n\). Then \(V\) becomes an \(\mathbb{F}_q[T]\)-module where \(T\) acts via \(N\), \(Tv=Nv\). Such modules are isomorphic to the module \(V_\lambda=\oplus_{j} \mathbb{F}_q[T]/(T^{n_j})\), for some partition \(\lambda = [n_1,\dots,n_k]\), \(n_1+\dots+n_k=n\) which is unique by the structure theorem of finitely generated modules over principal ideal domains. To show the partition \(\lambda\) associated to \(N\) we will use the notation \(N=N_\lambda\). 

Note that the the dual partition \(\lambda'\) arises from considering \(d_i = \dim(\ker(N^i))-\dim(\ker(N^{i-1}))\). To see this assume we switch to the matrix point of view and assume that $N\in M_n(\mathbb{F}_q)$ is  a nilpotent matrix over $\mathbb{F}_q$  with $r_j$ blocks of size $j\times j$ ($j=1,\dots, n_1$) in the Jordan normal form of $N$. Let \(d_i = \sum_{j \geq i} r_{j}.\) as above. Then it is easy to see that \(d_1\) is the number blocks, which also equals \(\dim \ker N\). The claim then follows from an easy inductive argument. One can alternatively define 
\[
d_i=r_i+r_{i+1}+\dots+r_k=\dim(\ker(N)\cap\mathrm{Im}(N^{i-1})).
\]

Finally we will need the order of the centralizer of unipotent elements in $GL_n(\mathbb{F}_q)$. Note that the centralizer of the unipotent $I+N$ is the same as the centralizer of the nilpotent transformation $N$.

\begin{pro}\label{prop:centralizer}
Let \(N=N_\lambda=[n_1,...,n_k]\) with dual partition \(\lambda'=[d_1,...,d_{n_1}]\). Then the centralizer of $N$ has cardinality
\[
\# Z_{GL_n(\mathbb{F}_q)}(N)=
\left(\prod_{j=1}^{n_1} (q^{r_j}-1)(q^{r_j}-q)\dots(q^{r_j}-q^{r_j-1})\right) \cdot q^{\sum_{j=1}^{n_1} (d_j^2-r_j^2)}=q^{\sum_{j=1}^k d_j^2} \prod_{j=1}^k\prod_{t_j=1}^{r_j}\left(1-\frac{1}{q^{t_j}}\right)\ .
\]
\end{pro}

\begin{proof}
This is Corollary IV.I.8 in \cite{Springer-Steinberg}.
\end{proof}
\begin{rem}
Note that if we define \(\phi_r(T)=\prod_{j=1}^r (1-T^j)\) 
and let \( \phi_\lambda(T)=\prod \phi_{r_i(\lambda)}(T)\)
 then the statement can be rewritten as 
 \[
\# Z_{GL_n(\mathbb{F}_q)}(N)=q^{\sum_{j=1}^k d_j^2} \,
\phi_\lambda\left( 1/q\right).
\]
\end{rem}

\subsection{Sylvester's equation}\label{subsec:Sylvester}\quad 

Assume that \(A\) is \(m\times m\), \(B\) is \(n\times n\) and \(X\) and \(C\) are \(m\times n\) matrices. 
The matrix equation 
\[
AX-XB=C,
\]
called Sylvester's equation \cite{Sylvester}, has a rich literature over the real, or complex fields in view of the important role it plays in various applications. (See. e.g. \cite{Duan}.) 
There are two important questions here, existence of solutions, and a description of all solutions \(\left\{ X \mid AX-XB=C\right\}\).

For our task of estimating \(K_n(A,B;p^k)\) we will concentrate on estimating the number of solutions. If the field of coefficients is \(\mathbb{F}_q\) for some $p$-power $q$, then the number of solutions is clearly either \(0\) or \( q^d \), where 
\begin{equation}\label{eq:triv-bound-Sylvester}
  d = d_{A,B}=\dim \{ X\,|\,AX-XB=0 \}\leq mn.  
\end{equation}

While the bound by \(mn\) is trivial, in the case when \( A=\lambda I_n,B=\lambda I_m \) for the same scalar \(\lambda\), one has \(d_{A,B}= mn\).

Note that we may interpret the equation via linear transformations. To do so let \( W=\mathbb{F}_q^m, V=\mathbb{F}_q^n\)  viewed as column vectors. Both \(W\) and \(V\) become \(\mathbb{F}_q[T]\)-modules via mapping \(T\) to \(A\) and \(B\) respectively.

If \(AX=XB\) then \(g(A)X=Xg(B)\) for any polynomial \(g \in \mathbb{F}_q[T]\) and so 
\(X\) gives rise to a module homomorphism from \(W\) to \(V\) which we denote \( \hom_{\mathbb{F}_q[T]}(W,V)\). 

For an irreducible polynomial \(f \in \mathbb{F}_q[T]\) let 
\[
V_{f^e}=\{ v\in V\mid f^e(A)v=0\}, 
\]
and similarly for \(W\). The \(f\)-primary component of \(V\) is \(\cup_{e=1}^\infty V_{f^e}\), which we denote by \( V_{f^\infty} \). 

Clearly if \(X\in \hom_{\mathbb{F}_q[T]}(W,V)\) then 
\begin{equation}\label{eq:intertwine}
 X(W_{f^e})\subset V_{f^e}   
\end{equation} which implies
\[
\hom_{\mathbb{F}_q[T]}(W,V) = \bigoplus_f \hom_{\mathbb{F}_q[T]}(W_{f^\infty},V_{f^\infty})
\]
the sum over \(f \in \mathbb{F}_q[T]\) irreducible. Since the problem is linear, we may go to a finite field extension if needed and then assume that the eigenvalues of \( A,B\) are in \(\mathbb{F}_q\). For \(f=T-\lambda\) and $e\geq 1$ we let
\begin{equation}\label{eq:d_infty-def}
d_e(A-\lambda) = \dim V_{(T-\lambda)^e}=\dim \ker (A-\lambda I)^e,\ d_\infty(A-\lambda)=d_n(A-\lambda).
\end{equation}
As an immediate corollary of the trivial bound (\ref{eq:triv-bound-Sylvester} we get the following
\begin{equation}
\label{eq:Sylvester-triv-bound-2}
d_{A,B}\leq \sum_{\lambda\in \mathbb{F}_q} 
d_\infty(A-\lambda)
d_\infty(B-\lambda).
\end{equation}
Note that that for \(A\) semisimple the inequality above becomes an equality, showing that the bound is sharp.

The application for us involves the special case when \(B=-A\). If \(A\) is invertible
the above bound is sufficient, but the nilpotent case needs a more refined version given in the following lemma.

\begin{lem}\label{lem:Sylvester-nilpotent}
Assume that \(A\in M_m(\mathbb{F}_q),B\in M_n(\mathbb{F}_q)\) are nilpotent. Let 
\[
k=\dim \ker (A),  
\quad l=\dim \ker (B), 
\]
Then 
\[
\dim \{ X \in M_{m,n}(\mathbb{F}_q) \mid AX=XB \} \leq  \frac{kn+ml}{2} .
\]
\end{lem}

\begin{proof}
As above let \( W=\mathbb{F}_q^m, V=\mathbb{F}_q^n\)  viewed  as \( \mathbb{F}_q[T]\) modules via mapping \(T\) to \(A\) and \(B\). Then 
\[
W\simeq \bigoplus_{i=1}^k \mathbb{F}_q[T]/(T^{m_i}), \quad V\simeq \bigoplus_{j=1}^l \mathbb{F}_q[T]/(T^{n_j}).
\]
for some partitions \((m_1,...,m_k), (n_1,...,n_l)\) of \(m\) and \(n\) respectively (with \(k,l\) as defined above).

Note that any element \(X\) of  
\( \hom_{\mathbb{F}_q[T]}(\mathbb{F}_q[T]/(T^b),  \mathbb{F}_q[T]/(T^a))\) is determined by the value \(X\) on \(1 \Mod T^b\) and so by (\ref{eq:intertwine}) 
\begin{equation}\label{eq:dim=min}
\dim \hom_{\mathbb{F}_q[T]}(\mathbb{F}_q[T]/(T^b),  \mathbb{F}_q[T]/(T^a)) = \min(a,b)
\end{equation}
This gives
\[
\dim \{ X \in M_{m,n}(\mathbb{F}_q) \mid AX=XB \} =\sum_{i,j=1}^{k,l} \min(m_i,n_j) 
\leq 
\sum_{i,j=1}^{k,l} \frac{m_i+n_j}{2} \leq \frac{kn+ml}{2}.
\]
\end{proof}

\begin{cor}
The trivial bound (\ref{eq:Sylvester-triv-bound-2}) can be strengthened to
\begin{equation}\label{eq:Sylvester-improved-bound}
\dim \{ X \Mod p\mid AX+XA =0\} \leq 
\sum_{\lambda\in \mathbb{F}_q } d_1(A-\lambda)
d_\infty(A+\lambda) 
\ .
\end{equation}

\end{cor}

\subsection{Generalities on multivariable Gauss sums.}\quad

We start with a general setup on \(V=\mathbb{F}_p^m\). Let \(F(x)=Q(x)+L(x)\) with a quadratic form \(Q\), and a linear form \(L\) on \(V\) and define 
\[
G(F;p)=\sum_{x\in V} e( F(x)/p)
\]
where \(e(z)=e^{2\pi i z}\). If \(L=0\) and \(p\neq 2\), the sum \(G(Q;p)\) is easy to evaluate after diagonalizing \(Q\); it is a product of trivial factors and Gauss sums. The case when \(p=2\) is slightly more involved \cite{Milnor}, but still explicit. 

For our use in what follows, some of the details are relevant, and so we sketch these. Assume that  \(p\neq 2\) when we have that \(Q\) comes from a bilinear form 
\[
B(x,y)=Q(x+y)-Q(x)-Q(y)
\] 
so that \(Q(x)=\frac{1}{2}B(x,x)\). \(B\) gives rise to a Riesz map \(R: V \to V^*\), \(R : y \mapsto (R_y: x \mapsto B(x,y))\), which may not be surjective if \(B\) is degenerate. Still we have the following dichotomy.

\begin{pro}\label{pro:gauss}
Assume that \(p\neq 2\). Let \(F(x)=Q(x)+L(x)\), where \(Q\) is a quadratic and \(L\) is a linear form on $V$.
\begin{enumerate}
    \item If \(L\) has a Riesz-representative, i.e. \(L=R_y\) for some \(y\), then \(G(F;p)= e(-Q(y)/p) G(Q;p) \).

\item If \(L\) does not have a Riesz-representative then \(G(F;p)=0\).

\end{enumerate}
\end{pro} 

\begin{proof} 
The first statement is trivial. To see the second, note that \(L\) does not have a Riesz representative if and only if \(\ker L\) does not contain \[V^\perp =\{ v \in V | B(v,y)=0 \text{ for all } y \in V\}. \]
Therefore there is \(y\), such that \(L(y)\neq 0\) but \(B(x,y)=0\) for all \(x\), and so that \(Q(x+y)=Q(x)+Q(y)\), and in particular \(Q(y)=0\). But then
\[
G(F;p)=\sum_x e(F(x+y)/p)=\sum_{x} e((Q(x+y)+L(x+y))/p)=
e(L(y)/p)G(F;p) 
\]
showing $G(F;p)=0$ since \(e(L(y)/p)\neq 1\).
\end{proof} 

The folllowing is an easy corollary of the evaluation of the standard Gauss sum \cite{Gauss-sums}.
\begin{cor}\label{cor:Gauss-sum}
Assume that \(F=Q+L\) and that \(L\) has a Riesz representative, \(L=R_y\). 
Write \(Q=Q_0\perp Q_1\) where \(Q_0\) is totally isotropic, and \(Q_1\) non-degenarate. Let \(\dim Q_1=r\). 
Then 
\[
G(F)=\left(\det Q_1/p\right)e(-Q(y)/p)p^{n-r/2},
\] where \( \left({\cdot}/{p}\right)\) is Legendre's symbol.  
\end{cor}

\section{The equation \(AX\equiv X^{-1}B\) to prime modulus}\label{sec:counting-bound-mod-p}

\subsection{Preliminary observations}\hfill

We will concentrate on the case when the equation \(AX=X^{-1}B\) is solvable. As noted above we may then simply assume that \(C=AX_0=X_0^{-1}B\) for some fixed solution $X_0$ and consider the equation \(CX=X^{-1}C\). For this equation we start with the proof of Proposition~\ref{prop:primer-decomp}.
\begin{proof}[Proof of Proposition~\ref{prop:primer-decomp}]
For any solution $X$ of the equation $CX=X^{-1}C$ we also have $XC=CX^{-1}$ by multiplying from the left by $X^{-1}$ and from the right by $X$. So we compute  $C^2X=C(CX)=CX^{-1}C=XC^2$, ie.\ $X$ and $C^2$ commute. Let us write $m_{C^2}(x)\in\mathbb{F}_p[x]$ for the minimal polynomial of $C^2$ and write it as $m_{C^2}(x)=\prod_{j=1}^r f_j(x)^{k_j}$ where $f_j(x)\in \mathbb{F}[x]$ is irreducible. 

Then we may decompose $\mathbb{F}_p^n=V_1\oplus \dots \oplus V_r$ as the direct sum of generalized eigenspaces where $V_j:=\ker(f_j(C^2)^{k_j})$. Since $X$ and $C^2$ commute, $V_j$ is an $X$-invariant subspace in $\mathbb{F}_p^n$ for $1\leq j\leq r$: indeed, for any $v\in V_j$ we have $f_j(C^2)^{k_j}v=0$ whence $f_j(C^2)^{k_j}(Xv)=Xf_j(C^2)^{k_j}v=0$ showing $Xv\in \ker(f_j(C^2)^{k_j})=V_j$. By a similar argument we also deduce that $V_j$ is also $C$-invariant (as $C$ also commutes with $C^2$). Therefore restricting the identity $CX=X^{-1}C$ to the subspace $V_j$ we deduce $C_jX_j=X_j^{-1}C_j$ where $C_j$ (resp.\ $X_j$) is the restriction of $C$ (resp.\ of $X$) to $V_j$. On the other hand, whenever we have matrices $X_j\in GL(V_j)$ with $C_jX_j=X_j^{-1}C_j$ ($1\leq j\leq r$) then we may form the block matrix $X$ from the matrices $X_j$ to obtain a solution of the equation $CX=X^{-1}C$. So the number of solutions of the equation $CX=X^{-1}C$ is the product of the number of solutions on each $V_j$. 

This proves the proposition.
\end{proof}

Now assume \(C\) is invertible. 
Since \(X\) is also invertible \(XCX=C\) if and only if \(Y=CX\) satisfies \(Y^2=C^2\) and we will count the number of solutions to this equation under the assumption that \(C^2\) has minimal polynomial \(f^k(x)\), with \(f(x)\neq x\) irreducible. For this recall that any linear transformation \(T\) has a unique multiplicative Jordan–Chevalley decomposition as 
\begin{equation}\label{eq:Chevalley-Jordan}
T=T_sT_u =T_uT_s
\end{equation}
where \(T_s\) is semisimple, \(T_u\) is unipotent. 
If  \(Y=Y_sY_u\) and \(Y^2=C^2\) then
\[
Y_s^2=C_s^2 \text{ and } Y_u^2=C_u^2.
\]

In case $p\neq 2$, we can immediately infer that \(Y_u=C_u\) from the following
\begin{lem}\label{lem:unipotent-unique}
Assume that \(Z_1,Z_2\in GL_n(\mathbb{F}_p)\) are unipotent such that 
\(Z_1^2=Z_2^2\). If \(p \neq 2\) then \(Z_1=Z_2\).
\end{lem}
\begin{proof}
For simplicity we use the simple property that for any unipotent element \(Z\) in \( GL_n(\mathbb{F}_p)\) we have \(Z^{p^r}=I\) for some \(r\). Since \(p\) is odd
\[
Z_1=(Z_1^2)^{(p^r+1)/2}=(Z_2^2)^{(p^r+1)/2}=Z_2.
\]
\end{proof}

Therefore estimating \( N^*(A,B;p) \) is reduced to estimating 
\begin{equation}\label{eq:n(C,C)-def}
    n(C,C;p)=\{ Y \in M_n(\mathbb{F}_p) \mid Y^2=C^2, YC_u=C_uY \}. 
\end{equation}

Our last observation is now the following 
\begin{lem}\label{Fq-linear-structure}
Put $V:=\mathbb{F}_p^n$ and assume that the minimal polynomial of $C^2\colon V\to V$ is $m_{C^2}(x)=f^k(x)$ where $f\in\mathbb{F}_p[x]$ is irreducible and let $q=p^{\deg f}$. Then $V$ has the structure of an $\mathbb{F}_q$-vectorspace such that all $C,X,Y,Y_s,Y_u\colon V\to V$ as above are $\mathbb{F}_q$-linear for any invertible solution $X$ of the equation $XCX=C$. Further, $Y_s^2$ is an $\mathbb{F}_q$-scalar multiple of the identity.
\end{lem}
\begin{proof}
We assume that  \(f(x)\neq x\) otherwise the claim is trivial. Note that \( \mathbb{F}_p[x]/(f)\) is  isomorphic to the field \(\mathbb{F}_q\)  with \(q=p^{\deg f}\) elements. We choose such an isomorphism and let \(\alpha\) denote the image of \(x\) in \(\mathbb{F}_q\), so that \(\mathbb{F}_q=\mathbb{F}_p[\alpha]\).
The ring \(\mathbb{F}_p[x]/(f^k)\) is then isomorphic to \(\mathbb{F}_q[t]/((t-\alpha)^k)\) and since \(C_s^2\in \mathbb{F}_p[C^2]\) is a semisimple element, it may be identified with \(\alpha \in \mathbb{F}_q^*\). 

Therefore the action of \(\mathbb{F}_p[C_s^2]\) on \(V\) gives an \( \mathbb{F}_q\)-linear structure and since \(X,Y,Y_s,Y_u\) all commute with \(C_s^2\) they may be viewed as an \(\mathbb{F}_q\)-linear transformation. Finally, we have \(Y_s^2=C_s^2=\alpha I\).
\end{proof}



\subsection{The proof of Theorem~\ref{thm:quadr-matrix-explicit} in the invertible cases}

Let $V:=\mathbb{F}_p^m$ and $C\colon V\to V$ be an invertible $\mathbb{F}_p$-linear map such that we have $m_{C^2}(x)=f^k(x)$ for the minimal polynomial of $C^2$ with some irreducible polynomial $f(x)\in \mathbb{F}_p[x]$. By Lemma \ref{Fq-linear-structure} we even have an $\mathbb{F}_q$-linear structure on $V$ (with $q:=p^{\deg f}$) such that both $C$ and any solution $X$ to the equation $CX=X^{-1}C$ are $\mathbb{F}_q$-linear. Further, $C$ has the Jordan--Chevalley decomposition $C=C_uC_s=C_sC_u$ with $C_u$ unipotent and $C_s$ semisimple. At first assume $p\neq 2$ and let \(N\) be the nilpotent transformation \(C_u-I\). Then \(V\) becomes an \(\mathbb{F}_q[T]\)-module where \(T\) acts via \(N\), \(Tv=Nv\). Such modules are isomorphic to the module 
\begin{equation}\label{eq:V_lambda}
 V_\lambda=\oplus_{j} \mathbb{F}_q[T]/(T^{n_j}),   
\end{equation}
for some partition \(\lambda = [n_1,..,n_k]\), \(n_1+..+n_k=n\) which is unique by the structure theorem of finitely generated modules over PIDs.   To show the partition \(\lambda\) associated to \(N\) we will use the notation \(N=N_\lambda\).

By Lemmas \ref{lem:unipotent-unique} and \ref{Fq-linear-structure} it is enough to count elements in the set
\begin{equation}\label{eq:Set-of-square=roots}
\mathcal{R}_\alpha(\lambda)=\{ Y\in GL_n(\mathbb{F}_q) \,|\, Y^2=\alpha I, YN_\lambda=N_\lambda Y\},\end{equation}
ie.\ we have $N^*(A,B;p)=n(C,C;p)=\#\mathcal{R}_\alpha(\lambda)$ where $\alpha\in \mathbb{F}_q^*$ denotes the unique eigenvalue of $C^2$ as an $\mathbb{F}_q$-linear.


We start with the case when \( \left( \frac{x}{f(x)} \right) = 1\) when we have \(\alpha = \beta^2\) for some \(\beta \in  \mathbb{F}_q^*\). 
We have the following

\begin{lem}\label{bijection-decomposition}
Let $\mathcal{S}(\lambda)$ denote set of pairs $(U_+,U_-)$ of $\mathbb{F}_q[T]$-submodules $U_+,U_-\leq V$ such that $V=U_+\oplus U_-$. The maps
\begin{eqnarray*}
Y&\mapsto& (U_{Y,+}, U_{Y,-})\\
Y_U(u_++u_-)=\beta(u_+-u_-)&\mapsfrom& V=U_+\oplus U_-
\end{eqnarray*}
are inverse bijections between $\mathcal{R}_\alpha(\lambda)$ and $\mathcal{S}(\lambda)$. Here $u_+\in U_+$, $u_-\in U_-$ and $U_{Y,+}$ (resp.\ $U_{Y,-}$) denotes the $\beta$-eigenspace (resp.\ $-\beta$-eigenspace) of $Y$. 
\end{lem}
\begin{proof}
Since $p\neq 2$, $Y^2=\alpha I$ implies $Y$ is semisimple whence $V$ is the direct sum of the two eigenspaces of $Y$. Moreover, these eigenspaces are $N$-invariant, ie.\ they are $\mathbb{F}_q[T]$-submodules. Conversely, given such a decomposition $V=U_+\oplus U_-$, we have $Y_U^2=\alpha I$ and $Y_U$ commutes with $N$.
\end{proof}

By the theorem of elementary divisors, for any decomposition $V=U_+\oplus U_-$ of the $\mathbb{F}_q[T]$-module $V$, $\lambda$ is the sum of the multisets $\mu$ and $\nu$ where $\mu$ (resp.\ $\nu$) is the partition of $\dim U_+$ (resp.\ of $\dim U_-$) corresponding to the restriction of $N$ to $U_+$ (resp.\ to $U_-$). Therefore we may write $\mathcal{S}(\lambda)$ as the union of 
\[
\mathcal{S}(\lambda,\mu,\nu)=\left\{(U_+,U_-)\mid U_+,U_-\leq_{\mathbb{F}_q[T]}V, U_+\oplus U_-=V,U_+\cong\bigoplus_{j\in\mu}\mathbb{F}_q[T]/(T^j),U_-\cong\bigoplus_{j\in\nu}\mathbb{F}_q[T]/(T^j) \right\}
\]
where $\mu$ runs over the multisets included in $\lambda$ and $\nu=\lambda-\mu$ is the difference.

\begin{lem}\label{centralizer-transitive}
For any decomposition $\lambda=\mu+\nu$ the centralizer $Z_{GL_{|\lambda|}(\mathbb{F}_q)}(N)$ of $N$ acts transitively on the set $\mathcal{S}(\lambda,\mu,\nu)$.
\end{lem}
\begin{proof}
Assume we are given two decompositions $U_+\oplus U_-=V=U'_+\oplus U'_-$ in $\mathcal{S}(\lambda,\mu,\nu)$. Then we have the isomorphisms $U_+\cong \bigoplus_{j\in \mu}\mathbb{F}_q[T]/(T^j)\cong U'_+$ and $U_-\cong \bigoplus_{j\in \nu}\mathbb{F}_q[T]/(T^j)\cong U'_-$ of $\mathbb{F}_q[T]$-modules. Taking the direct sum of these two isomorphisms we obtain an automorphism $g\colon V=U_+\oplus U_-\to U'_+\oplus U'_-=V$. Being $\mathbb{F}_q[T]$-linear it means that $g$ lies in $Z_{GL_{|\lambda|}(\mathbb{F}_q)}(N)$ when viewed as an $\mathbb{F}_q$-linear transformation.
\end{proof}

\begin{pro} If \(\alpha \in (\mathbb{F}_q^*)^2\) then
\[n(C,C;p)=\#\cR_{\alpha}(\lambda)=\sum_{\lambda=\mu+\nu} 
\frac{\# Z_{GL_{|\lambda|}(\mathbb{F}_q)}(N_\lambda)}{\# Z_{GL_{|\mu|}(\mathbb{F}_q)}(N_\mu) \# Z_{GL_{|\nu|}(\mathbb{F}_q)}(N_\nu)}\ .
\]

\end{pro}
\begin{proof}
By Lemma \ref{bijection-decomposition} we obtain $\# \cR_{\alpha}(\lambda)=\# \mathcal{S}(\lambda)=\sum_{\lambda=\mu+\nu}\# \mathcal{S}(\lambda,\mu,\nu)$. The statement follows from Lemma \ref{centralizer-transitive} noting that the stabilizer of a given decomposition $V=U_+\oplus U_-$ in $\mathcal{S}(\lambda,\mu,\nu)$ equals 
\[
Z_{GL_{|\lambda|}(\mathbb{F}_q)}(N_\lambda)\cap (GL_{|\mu|}(\mathbb{F}_q)\times GL_{|\nu|}(\mathbb{F}_q))=Z_{GL_{|\mu|}(\mathbb{F}_q)}(N_\mu)\times Z_{GL_{|\nu|}(\mathbb{F}_q)}(N_\nu)\ .
\]
\end{proof}

Now assume \( \left( \frac{x}{f(x)} \right) = -1\), so we have \(\alpha = \beta^2\) for some \(\beta \in  \mathbb{F}_{q^2}^*\setminus \mathbb{F}_{q}^*\). 
Put $\sigma$ for the nontrivial element in $\operatorname{Gal}(\mathbb{F}_{q^{2}}/\mathbb{F}_{q})$. Then $\varphi=\sigma\otimes I\colon \mathbb{F}_{q^{2}}\otimes_{\mathbb{F}_q}V$ is a $\sigma$-semilinear map (ie.\ $\mathbb{F}_q[T]$-linear with $\varphi(\beta v)=\sigma(\beta)\varphi(v)=-\beta\varphi(v)$). Further, for a $\mathbb{F}_{q^2}[T]$-module homomorphism $f\colon \mathbb{F}_{q^{2}}\otimes_{\mathbb{F}_q}V\to \mathbb{F}_{q^{2}}\otimes_{\mathbb{F}_q}V$ there exists an $\mathbb{F}_q[T]$-module homomorphism $\widetilde{f}\colon V\to V$ with $f=1\otimes\widetilde{f}$ if and only if $f$ commutes with $\varphi$, ie.\ $f\circ\varphi=\varphi\circ f$. 
\begin{lem}\label{transitive-nonsplit}
Assume $\alpha\notin(\mathbb{F}^*_q)^2$. Then the centralizer $Z_{GL_{|\lambda|}(\mathbb{F}_q)}(N_\lambda)$ acts transitively on the set $\mathcal{R}_\alpha(\lambda)$ (by conjugation).
\end{lem}
\begin{proof}
Let $Y,Y'\in \mathcal{R}_\alpha(\lambda)$, ie.\ $Y,Y'\colon V\to V$ are both $\mathbb{F}_q[T]$-linear isomorphisms with $Y^2=\alpha I=Y'^2$. Put $U\leq \mathbb{F}_{q^2}\otimes_{\mathbb{F}_q}V$ (resp.\ $U'\leq \mathbb{F}_{q^2}\otimes_{\mathbb{F}_q}V$) for the $\beta$-eigenspace of $1\otimes Y$ (resp.\ of $1\otimes Y'$). Then $\varphi(U)$ (resp.\ $\varphi(U')$) is the $\sigma(\beta)=-\beta$-eigenspace of $Y$ (resp.\ of $Y'$). In particular, we have $U\oplus \varphi(U)=\mathbb{F}_{q^2}\otimes_{\mathbb{F}_q}V=U'\oplus \varphi(U')$. Now if $U\cong \bigoplus_j\mathbb{F}_{q^2}[T]/(T^{m_j})$ for some partition $\mu=[m_1,\dots,m_s]$ then we have the isomorphism
\[
\varphi(U)\cong \bigoplus_j\mathbb{F}_{q^2}[T]/(\sigma(T)^{m_j})\cong \bigoplus_j\mathbb{F}_{q^2}[T]/(T^{m_j})\cong U 
\]
of $\mathbb{F}_{q^2}[T]$-modules whence $\lambda=\mu+\mu$. Similarly, $U'\cong \varphi(U')$. By the structure theorem for finitely generated modules over the PID $\mathbb{F}_{q^2}[T]$, we must have $U\cong U'$ as $\mathbb{F}_{q^2}[T]$-modules, as well. Taking such an isomorphism $S\colon U\to U'$ we also define $S(\varphi(u)):=\varphi(S(u))$ on $\varphi(U)$ giving rise to an $\mathbb{F}_{q^2}[T]$-linear automorphism $S\colon \mathbb{F}_{q^2}\otimes_{\mathbb{F}_q}V=U\oplus\varphi(U)\to U'\oplus\varphi(U')=\mathbb{F}_{q^2}\otimes_{\mathbb{F}_q}V$ that satisfies $S(1\otimes Y)S^{-1}=1\otimes Y'$. Moreover, $S$ descends to a map $\widetilde{S}\colon V\to V$ (such that $S=1\otimes \widetilde{S}$) since it commutes with $\varphi$. Finally, $\widetilde{S}$ satisfies $\widetilde{S}Y\widetilde{S}^{-1}=Y'$ and lies in the centralizer of $N_\lambda$ as it is $\mathbb{F}_q[T]$-linear.
\end{proof}

\begin{pro}
Assume $\alpha\notin(\mathbb{F}^*_q)^2$. Then we have $\lambda=\mu+\mu$ for some partition $\mu$ and
\[
n(C,C;p)=\#\mathcal{R}_\alpha(\lambda)=\frac{\# Z_{GL_{|\lambda|}(\mathbb{F}_q)}(N_\lambda)}{\# Z_{GL_{|\mu|}(\mathbb{F}_{q^2})}(N_\mu)}\ .
\]
\end{pro}
\begin{proof}
By Lemma \ref{transitive-nonsplit} $\mathcal{R}_\alpha(\lambda)$ is the conjugacy class of $C_s$ in $Z_{GL_{|\lambda|}(\mathbb{F}_q)}(N_\lambda)$. A moments thought shows that we may define an \(\mathbb{F}_{q^2}\)-linear structure on \(V\) where \(C_s\) acts via multiplication by \(\beta\) where the $\mathbb{F}_{q^2}$-linear maps are exactly those which are $\mathbb{F}_q$-linear and commute with $C_\beta$. In particular, the centralizer of $C_\beta$ in $Z_{GL_{|\lambda|}(\mathbb{F}_q)}(N_\lambda)$ equals $Z_{GL_{|\mu|}(\mathbb{F}_{q^2})}(N_\mu)$.
\end{proof}

This leads to formula \eqref{eq:sqrt-count-of-nonsquare}.

\begin{cor}\label{cor:reg-estimate}
Assume that $C$ has minimal polynomial $f(x)^k$ where $f(x)\neq x$ is irreducible and $p\neq 2$. Then we have $n(C,C,p)<\frac{q^2+1}{q^2-1}q^{\lfloor\frac{n^2}{2\deg^2 f}\rfloor}=\frac{q^2+1}{q^2-1}p^{\deg f\lfloor\frac{n^2}{2\deg^2 f}\rfloor}\leq \frac{p^2+1}{p^2-1}p^{\frac{n^2}{2}}$ and if $k=1$ then there exists a constant $0<c(q)<1$ (with $\lim_{q\to\infty}c(q)=1$) such that  
\[
c(q)q^{\lfloor\frac{n^2}{2\deg^2 f}\rfloor}<n(C,C;p)<\frac{q^2+1}{q^2-1}q^{\lfloor\frac{n^2}{2\deg^2 f}\rfloor}\ .
\]
In particular
\[
n(C,C;p)\leq 2 p^{n^2/2\deg f}.
\]
\end{cor}
\begin{proof}
Since $n(C,C,p)$ is the number of square roots of $C_s^2$ commuting with $C_u$, the case $C_u=I$ gives an upper bound for the number of solutions in general. So we may assume $N=0$. Put $n_1:=\dim_{\mathbb{F}_q}V=\frac{n}{\deg f}$. Then in the split case we compute
\begin{align*}
n(C,C;p)=\sum_{j=0}^{n_1}\frac{\#GL_{n_1}(\mathbb{F}_q)}{\#GL_{j}(\mathbb{F}_q)\#GL_{n_1-j}(\mathbb{F}_q)}=\\
\sum_{j=0}^{n_1}\frac{(q^{n_1}-1)\dots (q^{n_1}-q^{n_1-1})}{(q^j-1)\dots (q^j-q^{j-1})(q^{n_1-j}-1)\dots (q^{n_1-j}-q^{n_1-j-1})}=\\
\sum_{j=0}^{n_1}q^{n_1^2-j^2-(n_1-j)^2}\frac{(1-1/q^{n_1})\dots (1-1/q)}{(1-1/q^j)\dots (1-1/q)(1-1/q^{n_1-j})\dots (1-1/q)}<\\
\sum_{j=0}^{n_1}q^{n_1^2-j^2-(n_1-j)^2}<q^{\lfloor\frac{n_1^2}{2}\rfloor}(1+\sum_{j=1}^\infty\frac{2}{q^{2j}})=\frac{q^2+1}{q^2-1}q^{\lfloor\frac{n_1^2}{2}\rfloor}\ .
\end{align*}
On the other hand, we have
\begin{align*}
n(C,C;p)=\sum_{j=0}^{n_1}q^{n_1^2-j^2-(n_1-j)^2}\frac{(1-1/q^{n_1})\dots (1-1/q)}{(1-1/q^j)\dots (1-1/q)(1-1/q^{n_1-j})\dots (1-1/q)}>\\
q^{n_1^2-\lfloor \frac{n_1}{2}\rfloor^2-\lceil\frac{n_1}{2}\rceil^2}(1-1/q^{n_1})\dots (1-1/q)>c(q)q^{\lceil\frac{n_1^2}{2}\rceil}
\end{align*}
with constant $c(q)=\prod_{j=1}^\infty(1-1/q^j)$ that clearly satisfies $\lim_{q\to\infty}c(q)=1$.
\end{proof}

Finally, assume $p=2$. Since the $2$-Frobenius is bijective on finite fields of characteristic $2$, $C_s^2$ has a unique square root $Y_s=C_s$. So we need to count the square roots of the unipotent matrix $C_u^2$ or equivalently the square roots of the nilpotent matrix $C_u^2+I=(C_u+I)^2$.
\begin{lem}\label{square-jordan-block}
Assume that $q$ is a power of $2$. For any integer $n>0$ we have the identification
\[
\mathbb{F}_q[T]/(T^n)\cong \mathbb{F}_q[T^2]/((T^2)^{\lfloor\frac{n}{2}\rfloor})\oplus \mathbb{F}_q[T^2]/((T^2)^{\lceil\frac{n}{2}\rceil})
\]
as $\mathbb{F}_q[T^2]$-modules.
\end{lem}
\begin{proof}
This amounts to the fact that the square of a nilpotent Jordan block of size $n$ splits into two blocks of size $\lfloor\frac{n}{2}\rfloor$ and $\lceil\frac{n}{2}\rceil$.
\end{proof}

\begin{pro}
Assume $q$ is a power of $2$. Then the number of solutions of the matrix equation $Y^2=C^2$ equals
\[
\sum_\mu\frac{\# Z_{GL_n(\mathbb{F}_q)}(N_\lambda^2)}{\# Z_{GL_n(\mathbb{F}_q)}(N_\mu)}
\]
where $\mu=[m_1,\dots,m_k]$ runs on the set of partitions such that 
\[
\left[\lfloor\frac{m_1}{2}\rfloor,\lceil\frac{m_1}{2}\rceil,\dots,\lfloor\frac{m_k}{2}\rfloor,\lceil\frac{m_k}{2}\rceil\right]=\left[\lfloor\frac{n_1}{2}\rfloor,\lceil\frac{n_1}{2}\rceil,\dots,\lfloor\frac{n_k}{2}\rfloor,\lceil\frac{n_k}{2}\rceil\right]\ .
\]

\end{pro}
\begin{proof}
By Lemma \ref{square-jordan-block} $N_\mu^2$ is similar to $N_\lambda^2$ if and only if  
\[
\left[\lfloor\frac{m_1}{2}\rfloor,\lceil\frac{m_1}{2}\rceil,\dots,\lfloor\frac{m_k}{2}\rfloor,\lceil\frac{m_k}{2}\rceil\right]=\left[\lfloor\frac{n_1}{2}\rfloor,\lceil\frac{n_1}{2}\rceil,\dots,\lfloor\frac{n_k}{2}\rfloor,\lceil\frac{n_k}{2}\rceil\right]\ .
\]
So for each such $\mu$ we are reduced to determine the cardinality of the fiber at $N_\lambda^2$ of the map 
\begin{eqnarray*}
\{\text{conjugacy class of }N_\mu\} &\to& \{\text{conjugacy class of }N_\lambda^2\}\\
M&\mapsto& M^2\ .
\end{eqnarray*}
However, all the fibers of the above map have the same cardinality by conjugation, so the number of solutions to the above equation is 
\[
\frac{\# \{\text{conjugacy class of }N_\mu\}}{\# \{\text{conjugacy class of }N_\lambda^2\}}=\frac{\# Z_{GL_n(\mathbb{F}_q)}(N_\lambda^2)}{\# Z_{GL_n(\mathbb{F}_q)}(N_\mu)}\ .
\]
\end{proof}

\subsection{The proof of Theorem~\ref{thm:quadr-matrix-explicit} in the nilpotent case}

\begin{lem}\label{X-invariant-Ker-Im}
Assume $XCX=C$ for some $C\in M_n(\mathbb{F}_p) $ and $X\in \mathrm{GL}_n(\mathbb{F}_p)$. Then for any integer $j\geq 1$ the subspaces $\ker(C^j)$ and $\mathrm{Im}(C^j)$ are $X$-invariant.
\end{lem}
\begin{proof}
Since $X$ is invertible we may write $CX=X^{-1}C$ and $XC=CX^{-1}$, so by induction on $j$ we deduce $C^jX=X^{(-1)^j}C^j$ and $XC^j=C^jX^{(-1)^j}$. Therefore if $v\in \ker(C^j)$ then we have $C^jXv=X^{(-1)^j}C^jv=0$, ie.\ $Xv\in \ker(C^j)$. On the other hand, we compute $XC^jw=C^jX^{(-1)^j}w\in \mathrm{Im}(C^j)$ for any $w\in \mathbb{F}_p^n$.
\end{proof}

\begin{pro}\label{formula-nilpotent}
Let $C\in M_n(\mathbb{F}_p) $ be a nilpotent matrix such that there are $r_j$ blocks of size $j\times j$ ($j=1,\dots,k$) in the Jordan normal form of $C$. Then the number of solutions of the equation $XCX=C$ in $X\in \mathrm{GL}_n(\mathbb{F}_p)$ equals
\[
\left(\prod_{j=1}^k(p^{r_j}-1)(p^{r_j}-p)\dots(p^{r_j}-p^{r_j-1})\right)\cdot p^{\sum_{j=1}^k(d_j^2-r_j^2)}=p^{\sum_{j=1}^k d_j^2}\prod_{j=1}^k\prod_{t_j=1}^{r_j}\left(1-\frac{1}{p^{t_j}}\right)\ ,
\]
where we put $d_i=r_i+r_{i+1}+\dots+r_k=\dim(\ker(C)\cap\mathrm{Im}(C^{i-1}))=\dim(\ker(C^i))-\dim(\ker(C^{i-1}))$.
\end{pro}
\begin{proof}
First of all, we have $C^k=0$ and $\dim_{\mathbb{F}_p}(\ker(C)\cap \mathrm{Im}(C^i))=\sum_{j=i+1}^k r_j$. By lemma \ref{X-invariant-Ker-Im} the flag $0\leq \ker(C)\cap \mathrm{Im}(C^{k-1})\leq\dots\leq \ker(C)\cap \mathrm{Im}(C^i)\leq \dots \leq\ker(C)\cap \mathrm{Im}(C)\leq \ker(C)$ in $\ker(C)$ must be $X$-invariant for any solution $X$ of the equation $XCX=C$. The set of such maps $X_1:=X_{\mid \ker(C)}\colon \ker(C)\to \ker(C)$ is the parabolic subgroup $P_{(r_1,\dots,r_k)}$ of $\mathrm{GL}_{r_1+\dots+r_k}(\mathbb{F}_p)$ of type $(r_1,\dots,r_k)$ which has cardinality
\[
\# P_{(r_1,\dots,r_k)}=\left(\prod_{j=1}^k(p^{r_j}-1)(p^{r_j}-p)\dots(p^{r_j}-p^{r_j-1})\right)\cdot p^{\sum_{1\leq i<j\leq k}r_ir_j}\ .
\]
\begin{lem}
For any $2\leq t\leq k$ and $X_1\in P_{(r_1,\dots,r_k)}$ the number of extensions of $X_1$ to a one-to-one linear map $X_t\colon \ker(C^t)\to \ker(C^t)$ satisfying 
\begin{enumerate}
\item[$(i)$] $XCX=C$ and 
\item[$(ii)_t$] $\ker(C^i)\cap \mathrm{Im}(C^j)$ is $X_t$-invariant for all $1\leq i\leq t$ and $0\leq j\leq k$
\end{enumerate}
equals
\[ 
\prod_{2\leq j\leq t}p^{r_j(r_1+\dots +r_k)+r_{j+1}(r_2+\dots +r_k)+\dots +r_k(r_{k-j+1}+\dots+r_k)}\ .
\]
\end{lem}
\begin{proof}
We proceed by induction on $t$. Assume we have a map $X_{t-1}\colon \ker(C^{t-1})\to \ker(C^{t-1})$ satisfying $(i)$ and $(ii)_{t-1}$ and pick an element $v\in \ker(C^t)\cap \mathrm{Im}(C^j)$ not lying in $\ker(C^{t-1})+\mathrm{Im}(C^{j+1})$. We need to choose $X_t v\in \ker(C^t)\cap \mathrm{Im}(C^j)$ so that $CX_tv=X_{t-1}^{-1}Cv$ is satisfied since $Cv\in \ker(C^{t-1})\cap \mathrm{Im}(C^{j+1})$ on which subspace the map $X_t^{-1}$ is already defined by $X_{t-1}^{-1}$ (as $X_{t-1}$ is one-to-one). Moreover, $X_{t-1}^{-1}Cv$ lies in $\ker(C^{t-1})\cap \mathrm{Im}(C^{j+1})$ by assumption $(ii)_{t-1}$. In particular, there exists a vector $w\in \mathrm{Im}(C^{j})$ such that $Cw=X_{t-1}^{-1}Cv$ and $w\in \ker(C^t)$ (as we have $C^tw=C^{t-1}X_{t-1}^{-1}Cv=0$). Further, $w$ is unique upto $\ker(C)\cap \mathrm{Im}(C^j)$, so the possible values of $X_tv$ is exactly $w+(\ker(C)\cap \mathrm{Im}(C^j))$ which has cardinality $\# (\ker(C)\cap \mathrm{Im}(C^j))=p^{r_{j+1}+\dots+r_k}$. Finally, we let $v$ run on the lift of a basis of the quotient space $(\ker(C^t)\cap \mathrm{Im}(C^j))/(\ker(C^{t-1})+\mathrm{Im}(C^{j+1}))$ for any $j=k-t,k-t-1,\dots,1,0$ (noting $\mathrm{Im}(C^{k-t+1})\subseteq \ker(C^{t-1})$) we deduce that the number of extensions of $X_{t-1}$ to a map $X_t\colon \ker(C^t)\to \ker(C^t)$ satisfying $(i)$ and $(ii)_t$ is
\begin{align*}
\prod_{j=0}^{k-t}\# (\ker(C)\cap \mathrm{Im}(C^j))^{\dim_{\mathbb{F}_p}(\ker(C^t)\cap \mathrm{Im}(C^j))/(\ker(C^{t-1})+\mathrm{Im}(C^{j+1}))}=\\
p^{r_t(r_1+\dots +r_k)+r_{t+1}(r_2+\dots +r_k)+\dots +r_k(r_{k-t+1}+\dots+r_k)}
\end{align*}
as we have $\dim_{\mathbb{F}_p}(\ker(C^t)\cap \mathrm{Im}(C^j))/(\ker(C^{t-1})+\mathrm{Im}(C^{j+1}))=r_{j+t}$.
\end{proof}
The statement follows from the above lemma by taking $t=k$: the number of solutions of $XCX=C$ in invertible $X$ equals 
\begin{align*}
\# P_{(r_1,\dots,r_k)}\prod_{2\leq j\leq k}p^{r_j(r_1+\dots +r_k)+r_{j+1}(r_2+\dots +r_k)+\dots +r_k(r_{k-j+1}+\dots+r_k)}=\\
\left(\prod_{j=1}^k(p^{r_j}-1)(p^{r_j}-p)\dots(p^{r_j}-p^{r_j-1})\right)\cdot p^{\sum_{j=2}^k(j-1)r_j^2+\sum_{1\leq i<j\leq k}2ir_ir_j}=\\
\left(\prod_{j=1}^k(p^{r_j}-1)(p^{r_j}-p)\dots(p^{r_j}-p^{r_j-1})\right)\cdot p^{\sum_{j=1}^k(d_j^2-r_j^2)} 
\end{align*}
as claimed.
\end{proof}

\begin{cor}\label{cor:nilp-estimate}
Assume $C$ is nilpotent. Then we have $n(C,C;p)\leq p^{\sum_{j=1}^kd_j^2}\leq p^{\mathrm{rk}(C)^2+(n-\mathrm{rk}(C))^2}$.
\end{cor}
\begin{proof}
Using Proposition \ref{formula-nilpotent} we compute
\begin{align*}
n(C,C;p)=\left(\prod_{j=1}^k(p^{r_j}-1)(p^{r_j}-p)\dots(p^{r_j}-p^{r_j-1})\right)\cdot p^{\sum_{j=1}^k(d_j^2-r_j^2)}\leq\\ \leq\left(\prod_{j=1}^k(p^{r_j})^{r_j})\right)\cdot p^{\sum_{j=1}^k(d_j^2-r_j^2)}
=p^{\sum_{j=1}^kd_j^2}\leq p^{d_1^2+(n-d_1)^2}=p^{\mathrm{rk}(C)^2+(n-\mathrm{rk}(C))^2}
\end{align*}
by noting $d_1=\dim\ker(C)=n-\mathrm{rk}(C)$ and $\sum_{j=1}^kd_j=n$.
\end{proof}
\begin{rem}
For fixed $n$ and $p\to\infty$ the above upper estimate $p^{\sum_{j=1}^k d_j^2}$ is in fact the order of magnitude of $n(C,C,p)$:
\begin{align*}
\frac{n(C,C;p)}{p^{\sum_{j=1}^k d_j^2}}=\prod_{j=1}^k\prod_{t_j=1}^{r_j}\left(1-\frac{1}{p^{t_j}}\right)>\left(1-\frac{1}{p}\right)^n\ .
\end{align*}
\end{rem}

\section{The proofs of the bounds}\label{sec:general-bounds}

\subsection{The equation \(AX\equiv X^{-1}B\) modulo prime powers}\quad

Assume that \(A,B\in M_n(\Z)\). We are interested in estimating the size of the affine variety \(V_{A,B}(p^l)\) where 

\begin{equation}\label{eq:V_{A,B}}
V_{A,B}(p^l)=\{ X\in GL_n(\mathbb{Z}/p^l\mathbb{Z})\mid AX = X^{-1}B\}.
\end{equation}

We will collect elements of 
\(V_{A,B}(p^{l+1})\) according to their image in \(V_{A,B}(p^l)\). The final push down to \(l=1\) will play a special role and we let
\begin{equation}\label{eq:V_{A,B}^C}
V_{A,B}^{(C)}(p^l)=\{ X\in V_{A,B}(p^l) \mid AX\equiv X^{-1}B \equiv C \Mod p\}.
\end{equation}

Let now \(X_0 \in V_{A,B}^{(C)}(p^l)\) be given. Then all \(X \in V_{A,B}(p^{l+1})\)  such that \(X\equiv X_0 \Mod p^l\) may be written as 
\(X=X_0(I+p^l Y)\), for some \(Y \Mod p\). The goal is to bound \(Y\) for which   (\ref{eq:V_{A,B}}) also holds  \(\Mod p^{l+1}\). This leads to
\begin{equation*}
AX_0 Y+Y X_0^{-1}B \equiv (X_0^{-1}B-AX_0)/p^l \Mod p.
\end{equation*}
Since \(X_0 \in V_{A,B}^{(C)}(p^l)\) we have that \(Y\) is a solution to the Sylvester equation
\begin{equation}\label{eq:sylvester}
CY+Y C \equiv (X_0^{-1}B-AX_0)/p^l \Mod p.
\end{equation}
Note that the equation above might have no solution, or exactly as many solution as 
\[
CY=-YC \Mod p
\]
for which we may apply Proposition~\ref{lem:Sylvester-nilpotent} and its corollary. This gives

\begin{lem}\label{lem:sylvester-triv-up} Let \( \rk C = r\), \(\rk_\infty C=r_\infty\). If \(l\geq 1\) and \(p\neq 2\) then
\[
\#V_{A,B}^{(C)}(p^{l+1}) \leq p^{(n-r)(n-r_\infty)+r_\infty^2/2}\#V_{A,B}^{(C)}(p^{l}).
\]



\end{lem}

\begin{proof}
This is merely a restatement of Lemma~\ref{lem:Sylvester-nilpotent} and (\ref{eq:Sylvester-triv-bound-2}). First note that \[
\dim \ker C =n-r \text{ and } \dim \ker C^n=n-r_\infty.
\]

To simplify the contribution of the non-zero eigenvalues in (\ref{eq:Sylvester-triv-bound-2}) use that 
\[
2 d_\infty(A-\lambda)d_\infty(A+\lambda)
\leq  
\frac{
\bigl( d_\infty(A-\lambda)+d_\infty(A+\lambda)
\bigr)^2}{2}
\]
and
that for \(a_1,...,a_k\) positive integers, \( a_1^2+...+a_k^2\leq \left( a_1+...+a_k \right)^2 \).
\end{proof}

This estimate is wasteful since the solution set could be empty. However, this will suffice for us.

\begin{proof}[The proof of Theorem~\ref{thm:quadr-matrix-bound}] When \(l=1\) the bound for \(N^*(C,C;p)\) follows from Theorem~\ref{thm:quadr-matrix-explicit} together with exact formulae in Proposition~\ref{prop:centralizer} which gave Corollaries \ref{cor:reg-estimate} and \ref{cor:nilp-estimate}. 

To see this, note that we may decompose \(C\) (over the ground field \(\mathbb{F}_p\)) as a block matrix with one block invertible of size \(r_\infty\) and one block nilpotent of size \( n-r_\infty \). 

For the invertible part we have the upper bound \( 2^{r_\infty} p^{r_\infty^2/2} \) using Corollary \ref{cor:reg-estimate} for each irreducible factor $\neq X$ of the minimal polynomial of $C^2$, noting that there at most $r_\infty$ such factors, and applying \( a_1^2+...+a_k^2\leq \left( a_1+...+a_k \right)^2 \) for positive integers \(a_1,...,a_k\). The nilpotent block has rank \(r-r_\infty\) and so by Corollary \ref{cor:nilp-estimate} we have the upper bound 
\( p^{(n-r)^2 + (r-r_\infty)^2} \) from which
\begin{equation}\label{eq:thm1.6_l=1}
N^*(C,C;p) \leq 2^{r_\infty} p ^{(n-r)^2 + (r-r_\infty)^2 + r_\infty^2/2} .    
\end{equation}
This proves the first statement in case \(l=1\). Whenever \( l>1\) we use Lemma~\ref{lem:sylvester-triv-up} inductively to get that
\[
N^*(C,C;p^l)\leq 
2^{r_\infty} p^{ e(l,n,r,r_\infty)}
\]
where
\[
e(l,n,r,r_\infty)= (n-r)^2 + (r-r_\infty)^2 + r_\infty^2/2 +(l-1)\left((n-r)(n-r_\infty) + r_\infty^2/2\right).
\]

In order to prove the second statement, note that unless \(N^*(A,B;p^l)=0\), we find a common value \(C:=AX_0=X_0^{-1}B\) such that \(N^*(A,B;p^l)=N^*(C,C;p^l)\) and put \(r_\infty:=r_\infty(C)\). Moreover, from Lemma~\ref{lem:unipotent-unique} the value of \(r_\infty\) is the same for any of the 
\( C \)-s that arise. So we are bound to estimate \(e(l,n,r,r_\infty)\).

To simplify the exponent assume first that \( n/2\leq r \leq n \). A calculation shows that the maximum of the function
\[
(1-x)^2 + (x-y)^2 + y^2/2 +(l-1)\left( (1-x)(1-y) + y^2/2 \right)
\]
on the domain 
\begin{equation}\label{eq:D-domain-def}
D=\{(x,y) \mid 1/2\leq x \leq 1,\; 0 \leq y \leq x\} \  .
\end{equation}
is \(l/2\), proving the claim in this case.

For \(r< n/2\) we use that \( (n-r)(n-r_\infty) + r_\infty^2/2 \leq n(n-r)\), and so
\[
e(l,n,r,r_\infty) \leq ln(n-r)
\]
in view of \((n-r)^2+r^2-n(n-r)=r(2r-n)\leq 0.
\)

This establishes both bounds in (\ref{thm:1.6-part2}). 

Finally to prove the universal bound \( N^*(A,B;p^l)\leq 2^np^{l(n^2-n)}\) note that it holds trivially for \(l\geq 2\), since \(n^2-n\geq n^2/2\). 

For \(N^*(A,B;p)\) start with the bound in (\ref{eq:thm1.6_l=1}). Note that if \(r=n\) then also \(r_\infty=n\) and so it is enough to prove that, for \( 0 \leq r_\infty \leq r\), \(1 \leq r \leq n-1\), we have 
\[
(n-r)^2 + (r-r_\infty)^2 + r_\infty^2/2 \leq n^2-n.
\]
However since \(0\leq r_\infty\leq r\) we have
\[
 (r-r_\infty)^2 + r_\infty^2/2 \leq r^2 \ ,  
\]
and for \(1 \leq r \leq n-1\)
\[
(n-r)^2+r^2 \leq n^2-2n+2. 
\]
Finally \(n^2-2n+2 \leq n^2-n \) holds since \(n\geq 2\) .
\end{proof}

\subsection{Gauss sums of matrices}

There are various ways exponential sums with quadratic functions of the entries of an \(n \times n\) matrix arise. For example in the theory of Siegel modular forms \(Q(X)=\Tr X^tAX\), and the associated Gauss sums play an important role see e.g \cite{Lynne}. These have a very different flavor than ours, as the tensor properties allow one to diagonalize \(A\), which immediately yields a diagonalization of the quadratic form \(Q(x_{11},x_{12},\dots,x_{nn})\).
This approach is not directly applicable to our situation since we have \(Q(X)= \Tr TX^2\) for some matrix \(T\). While this case appeared in the literature, see e.g. \cite{Kuroda} our treatment is based directly on Proposition~\ref{pro:gauss} and its corollary \ref{cor:Gauss-sum}.

\begin{proof}[The proof of Proposition~\ref{pro:matrix-gauss}]
We have to estimate the sum
\[ 
S_{A,B}(X;p) = \sum_{U \Mod p} \psi((SU+TU^2)/p), 
\] 
where \(X\) is such that \(AX\equiv X^{-1}B \Mod p^l\), and where \(S=(AX-X^{-1}B)/p^l\) and \(T=AX \Mod p\). This is clearly a general Gauss sum. To apply Corollary \ref{cor:Gauss-sum} let 
\(B(U,Y)=Q(U+Y)-Q(U)-Q(Y)=\Tr((TY+YT)U )\) be the associated bilinear form.

We have that either \(S_{A,B}(X;p)=0\) or there exists \(Y\) such that \(\Tr SU= B(U,Y) \) for some \(Y\) in which case
\[
S_{A,B}(X;p)=\left(\frac{\det(Q_1)}{p}\right) e^{-2\pi iQ(Y)/p} p^{n^2-R}\ , 
\]
where \(R\) is the rank of the quadratic form \(Q(X)=\Tr (TX^2)\). 

Whether there exists \(Y\) such that $\Tr SU=B(U,Y)$ for all $U$ is again determined by the solubility of a Sylvester  equation 
\[
TY+YT=S.
\]
Moreover it implies that the rank of \(Q\) is \(R=n^2-K\), where \(K=\# \{ Y: TY+YT\equiv 0 \mod p\}\). This is estimated as in Lemma \ref{lem:sylvester-triv-up} using Lemma~\ref{lem:Sylvester-nilpotent} which gives the claim. 
\end{proof}

\subsection{Bounding \(K_n(A,B;p^k)\)}

\begin{proof}[The proof of Theorem \ref{thm:main}] The case of \(k=1\) was handled in \cite{E-T}. 

When \(k=2l\) we have $|K_n(A,B;p^k)|\leq p^{ln^2}N^*(A,B;p^l)$ by Proposition \ref{pro:red} noting that any mod $p^l$ solution of the matrix equation $AX\equiv X^{-1}B\pmod{p^l}$ has exactly $p^{ln^2}$ lifts to  $GL_n(\mathbb{Z}/p^k\mathbb{Z})$. 

Similarly, in case \(k=2l+1\) ($l\geq1$) we deduce $|K_n(A,B;p^k)|\leq p^{ln^2}N^*(A,B;p^l)\max_X|S_{A,B}(X;p)|$ from Proposition \ref{pro:red}. The estimate for the Gauss sum $S_{A,B}(X;p)$ is given in Proposition \ref{pro:matrix-gauss} while the estimate for $N^*(A,B;p^l)$ is in Theorem \ref{thm:quadr-matrix-bound} both in the general case and under the stronger assumption $\gcd(\det A, \det B,p)=1$.
\end{proof}


\begin{thebibliography}{99}
\bibitem{Gauss-sums} Berndt, Bruce C., Kenneth S. Williams, and Ronald J. Evans. "Gauss and Jacobi sums." \emph{Wiley} (1998).

\bibitem{Cochrane-Zheng} Cochrane, Todd, and Zhiyong Zheng. "A survey on pure and mixed exponential sums modulo prime powers.", in: \emph{Surveys in Number Theory} Crc Press (2009).

\bibitem{Dabrowksi-Fisher} Dabrowski, Romuald, and Benji Fisher. "A stationary phase formula for exponential sums over $\mathbb{Z}/p^{m} \mathbb{Z} $ and applications to \(GL (3)\)-Kloosterman sums." \emph{Acta Arith.} \textbf{80}(1) (1997), 1--48.

\bibitem{Duan} Duan, Guang-Ren, "Generalized Sylvester equations: unified parametric solutions", Crc Press (2019).

\bibitem{E-B-L-S} El Baz, Daniel, Lee, Min and Strömbergsson, Andreas, "Effective equidistribution of rational points on expanding horospheres", \href{https://arxiv.org/abs/2212.07408}{arXiv:2212.07408}

\bibitem{E-T} Erdélyi, Márton and Tóth, Árpád, "Matrix Kloosterman sums", to appear in \emph{Algebra \& Number Theory}, \href{https://arxiv.org/abs/2109.00762}{arXiv:2109.00762}

\bibitem{Kuroda} Kuroda, Mitsuru. "Quadratic Gauss sums on matrices." \emph{Linear algebra and its applications} \textbf{384} (2004), 187--198.

\bibitem{Milnor}
Milnor, John Willard, and Dale Husemoller, "Symmetric bilinear forms", Vol. \textbf{73}, Berlin, Springer (1973).

\bibitem{Lee-Marklof} Lee, Min and Marklof, Jens, "Effective equidistribution of rational points on expanding horospheres", \emph{International Mathematics Research Notices} 2018.21 (2018), 6581--6610.


\bibitem{Morrison} Morrison, Kent E., "Integer Sequences and Matrices Over Finite Fields", \emph{Journal of Integer Sequences} \textbf{9}(2) (2006), 3.


\bibitem{Penrose} Penrose, Roger, "A generalized inverse for matrices", \emph{Mathematical proceedings of the Cambridge philosophical society} \textbf{51}(3), Cambridge University Press (1955).


\bibitem{Potter} Potter, James E., "Matrix quadratic solutions", \emph{SIAM Journal on Applied Mathematics} \textbf{14}(3) (1966), 496--501.


\bibitem{Rosen} Rosen, Michael, "Number theory in function fields", Vol. \textbf{210}, Springer Science \& Business Media (2002).

\bibitem{Salie} Salié, Hans. "Über die Kloostermanschen Summen \(S(u,v;q)\)." \emph{Mathematische Zeitschrift} \textbf{34}(1) (1932), 91--109.

\bibitem{Stanley} Stanley, Richard P., "Enumerative Combinatorics Volume 1 second edition", Cambridge studies in advanced mathematics (2011).

\bibitem{Springer-Steinberg} Springer, Thomas A. and Steinberg, Robert, "Conjugacy classes" in: \emph{Seminar on Algebraic Groups and Related Finite Groups} (The Institute for Advanced Study, Princeton, N.J., 1968/69), pages 167--266. Springer, Berlin, 1970. Lecture Notes in Mathematics, Vol. 131.

\bibitem{Sylvester} Sylvester, James Joseph. "Sur l’équation en matrices px= xq." \emph{CR Acad.\ Sci.\ Paris} \textbf{99}(2) (1884), 67--71.

\bibitem{Lynne} Walling, Lynne H. "Explicitly realizing average Siegel theta series as linear combinations of Eisenstein series." \emph{The Ramanujan Journal} \textbf{47}(3) (2018), 475--499.
\end{thebibliography}
\end{document}